\documentclass{amsart}
\usepackage{amsmath,amsfonts,amsthm,amssymb,verbatim,enumerate,graphicx}
\usepackage{color}

\usepackage[flushmargin]{footmisc}
\newtheorem{theorem}{Theorem}[section]
\newtheorem*{theoremA*}{Theorem A}
\newtheorem*{theoremB*}{Theorem B}
\newtheorem{lemma}[theorem]{Lemma}
\newtheorem{corollary}[theorem]{Corollary}
\newtheorem{example}{Example}
\newtheorem{question}{Question}
\newtheorem*{remark}{Remark}
\newtheorem*{remarks}{Remarks}
\numberwithin{equation}{section}
\def \isnatural {\in\mathbb{N}}
\def \iscomplex {\in\mathbb{C}}
\newcommand{\tef}{transcendental entire function}

\newcommand\qfor{\quad\text{for }}
\newcommand\sphere{\widehat{\mathbb{C}}}
\newcommand \C{\mathbb{C}}
\newcommand \N{\mathbb{N}}

\newcommand{\mconn}{multiply connected}
\newcommand{\sw}{spider's web}

\makeatletter
\def\blfootnote{\xdef\@thefnmark{}\@footnotetext}
\makeatother
\begin{document}
%
%
\title[On the set where iterates are neither escaping nor bounded]{On the set where the iterates of an entire function are neither escaping nor bounded}
\author{J. W. Osborne, \, D. J. Sixsmith}
\address{Department of Mathematics and Statistics \\
	 The Open University \\
   Walton Hall\\
   Milton Keynes MK7 6AA\\
   UK}
\email{john.osborne@open.ac.uk}
\address{Department of Mathematics and Statistics \\
	 The Open University \\
   Walton Hall\\
   Milton Keynes MK7 6AA\\
   UK}
\email{david.sixsmith@open.ac.uk}
%
%
\begin{abstract}
For a {\tef} $f$, we study the set of points $BU(f)$ whose iterates under $ f $ neither escape to infinity nor are bounded. We give new results on the connectedness properties of this set and show that, if $U$ is a Fatou component that meets $BU(f)$, then most boundary points of $U$ (in the sense of harmonic measure) lie in $BU(f)$. We prove this using a new result concerning the set of limit points of the iterates of $f$ on the boundary of a wandering domain. Finally, we give some examples to illustrate our results.
\end{abstract}
\maketitle
%
%
\blfootnote{The second author was supported by Engineering and Physical Sciences Research Council grant EP/J022160/1.}
\section{Introduction}
\label{sec1}
Denote the $ n $th iterate of an entire function $ f $ by $ f^n $, for $ n \in \mathbb{N}. $  The \emph{Fatou set} $F(f)$ is defined to be the set of points $z\iscomplex$ such that $\{f^n\}_{n\isnatural}$ is a normal family in some neighbourhood of $z$. A component of $F(f)$ is referred to as a \emph{Fatou component}. The \emph{Julia set} $J(f)$ is the complement of $F(f)$ in $\mathbb{C}$. The Fatou and Julia sets together form a fundamental dynamical partition of $\mathbb{C}$ \---\ roughly speaking, the dynamical behaviour of $ f $ is stable on the Fatou set and chaotic on the Julia set. For an introduction to the properties of these sets, see \cite{aB, MR1216719, Mil}. 

Here we work with an alternative partition of the plane, based on the nature of the orbits of points; the \emph{orbit} of a point $z$ is the sequence $(f^n(z))_{n\geq 0}$ of its iterates under $ f $. Orbits may tend to infinity (in which case we say that they escape), or they may be bounded, or they may neither escape nor be bounded.  This paper is concerned with the properties of the set of points whose orbits neither escape nor are bounded \---\ that is, points whose orbits contain both bounded and unbounded suborbits.  For an entire function $ f $, we denote this set by $ BU(f). $

The set $ I(f) $ of points whose orbits escape (the \emph{escaping set}) is defined by 
$$ I(f) = \{z : f^n(z)\rightarrow\infty\text{ as }n\rightarrow\infty\}. $$
For a non-linear polynomial $ P $, the escaping set $ I(P) $ is the basin of attraction of the point at infinity and so lies in the Fatou set.  The escaping set for a general \tef\ $ f $ was first studied by Eremenko \cite{MR1102727} who showed that, by contrast, $ I(f) \cap J(f) \neq \emptyset $.  He also showed that all components of $ \overline{I(f)} $ are unbounded, and conjectured that the same may be true of all components of~$ I(f) $.  This conjecture, which remains open, has been the focus of much subsequent research in complex dynamics. 

The set $K(f)$ of points whose orbits are bounded is defined by 
$$
   K(f) = \{z : \text{ there exists } R>0 \text{ such that } |f^n(z)| \leq R, \text{ for } n\geq 0\}.
$$
When $P$ is a non-linear polynomial, the set $K(P)$ is known as the \emph{filled Julia set} and has been extensively investigated. The set $K(f)$ for a {\tef} $f$ was studied in \cite{MR2869069} and \cite{2012arXiv1208.3692O}. 

We define the set $ BU(f) $ as follows:
$$
BU(f) = \mathbb{C}\setminus(I(f) \cup K(f)).
$$
We say that a set $S$ is \emph{completely invariant} if $z \in S$ implies that $f(z) \in S$ and that $f^{-1}(\{z\}) \subset S$. It is easy to see that $I(f), K(f)$ and $ BU(f) $ are each completely invariant and that together they form a dynamical partition of the plane.

For a non-linear polynomial $ P $, it is well known that $BU(P) = \emptyset$. However, if $ f $ is a \tef, then $ BU(f) $ always contains points in the Julia set~\cite[Lemma~1]{BD2}. Examples of a {\tef} with a Fatou component in $BU(f)$ have been given by Eremenko and Lyubich \cite{MR918638} and by Bishop \cite{Bish3}.  We are unaware of any studies of the properties of $ BU(f) $ for a general {\tef}. 

In the next section of this paper, we briefly review some basic properties of $ BU(f) $ and prove the following. Recall that a Fatou component $ U $ is called a \emph{wandering domain} if it is not eventually periodic \---\ that is, if $U_n$ is the Fatou component containing $f^n(U)$ for $n\isnatural$, then $U_n = U_m$ only if $n=m$. 
\begin{theorem}
\label{Tbasicresults}
Let $f$ be a {\tef}.
\begin{enumerate}[(a)]
\item If $U$ is a Fatou component of $f$ and $U \cap BU(f) \ne \emptyset$, then $U \subset BU(f)$ and $U$~is a wandering domain.
\item $J(f) = \partial BU(f)$.
\end{enumerate}
\end{theorem}

Our first main result gives certain connectedness and boundedness properties of $BU(f)$ and related sets. Recall that all components of $\overline{I(f)}$ are unbounded. The first author has shown in \cite{2012arXiv1208.3692O} that either $K(f) \cap J(f)$ is connected, or every neighbourhood of any point in $J(f)$ meets uncountably many components of $K(f) \cap J(f)$; for a certain class of functions, the same property was also shown to hold for $K(f)$. Related results have been proved for $ K(f)^c $ in \cite{Kfc}, and for $ I(f) $ by Rippon and Stallard (forthcoming work).  

For $ BU(f) $ we prove the following.
\begin{theorem}
\label{Tconnectnessandboundedness}
Let $f$ be a {\tef}.
\begin{enumerate}[(a)]
\item If $f$ has no multiply connected Fatou components, then all components of $\overline{BU(f)}$ are unbounded.  Otherwise, all components of $\overline{BU(f)}$ are bounded. 
\item Either $BU(f) \cap J(f) $ is connected, or every neighbourhood of any point in $J(f)$ meets uncountably many components of $BU(f) \cap J(f) $.
\item Either $BU(f)$ is connected, or every neighbourhood of any point in $J(f)$ meets uncountably many components of $BU(f)$.
\end{enumerate}
\end{theorem}
Next, we consider the boundary of a Fatou component that lies in $BU(f)$. Rippon and Stallard \cite[Theorem 1.1]{MR2801622} showed that, if $f$ is a {\tef} and $U\subset I(f)$ is a wandering domain, then $\partial U \cap I(f)^c$ has harmonic measure zero relative to $U$; we refer to Section~\ref{S3} for a definition. They also showed \cite[Theorem 1.2]{MR2801622} that, if $U$ is a Fatou component of $f$ and $\partial U \cap I(f)$ has positive harmonic measure relative to $U$, then $U \subset I(f)$. We prove the following version of these results for $BU(f)$.

\begin{theorem}
\label{TboundaryinBU}
Suppose that $f$ is a {\tef} and that $U$ is a Fatou component of $f$.
\begin{enumerate}[(a)]
\item If $U \subset BU(f)$, then $\partial U \cap BU(f)^c$ has harmonic measure zero relative to $U$.
\item If $U$ is a wandering domain and $\partial U \cap BU(f)$ has positive harmonic measure relative to $U$, then $U \subset BU(f)$.
\end{enumerate}
\end{theorem}

In view of Theorem~\ref{Tbasicresults}(a), the following corollary of Theorem~\ref{TboundaryinBU} is immediate.

\begin{corollary}
\label{CboundaryinBU}
If $f$ is a {\tef}, then every component of $ BU(f) $ meets $ J(f) $.
\end{corollary}

In fact, Theorem~\ref{TboundaryinBU} is a consequence of the following more general result which is of wider interest. For a \tef\ $ f $, the \emph{$\omega$-limit set} $\Lambda(z,f)$ of the point $ z \in \mathbb{C} $ is the set of accumulation points of its orbit in $\sphere$ (we avoid the usual notation $\omega(z,f)$ because of possible confusion with the notation for harmonic measure).  For a wandering domain $ U $ of $f$, it is well known -- see, for example, \cite[p.317]{Cremer} and \cite[Section 28]{MR1504797} -- that $\Lambda(z_1,f) = \Lambda(z_2,f)$ for $z_1, z_2 \in U$, so in this case we can write $\Lambda(U,f)$ without ambiguity. We show that, in a precise sense, this equality of $\omega$-limit sets extends to most of the boundary of a wandering domain.

\begin{theorem}
\label{Tboundaryingeneral}
Suppose that $f$ is a {\tef} and that $U$ is a wandering domain of $f$. Then the set
$$
  \{ z \in \partial U : \Lambda(z,f) \ne \Lambda(U, f) \}
$$
has harmonic measure zero relative to $U$.
\end{theorem}

We list below a number of further questions that arise naturally from the results and examples in this paper.

\begin{question}\normalfont
Does there exists a {\tef} $f$ such that $BU(f)$ is connected?
\end{question}
\begin{question}\normalfont
Is there a \tef\ $ f $ with a wandering domain $ U $ such that the $ \omega $-limit set $ \Lambda(U,f) $ is an uncountable set (see the remark at the end of Section~\ref{S3})? 
\end{question}
\begin{question}\normalfont
In view of Example~\ref{exampleunbdfatou} below, is there a \tef\ with an unbounded wandering domain in $ BU(f) $, all of whose iterates are unbounded?
\end{question}

The structure of this paper is as follows. In Section~\ref{S1}, we briefly review some basic properties of $ BU(f) $ and prove Theorem~\ref{Tbasicresults}.  Then, in Section~\ref{S2}, we prove the connectedness and boundedness properties of $ BU(f) $ and related sets given in Theorem~\ref{Tconnectnessandboundedness}. In Section~\ref{S3} we give the proofs of Theorem~\ref{TboundaryinBU} and Theorem~\ref{Tboundaryingeneral}. In Section~\ref{Sdimension} we briefly discuss the Hausdorff dimension of $BU(f)$. Finally, in Section~\ref{S4}, we give a number of examples to illustrate some of the different topological and dynamical structures that can occur for $BU(f)$. In particular, we show that there are {\tef}s for which $BU(f)$ is totally disconnected, or has uncountably many unbounded components with empty interior, or contains an unbounded Fatou component.
%
%
%
%
\section{Basic properties of $ BU(f) $}
\label{S1}

For the convenience of the reader, we briefly review a number of basic properties of $ BU(f) $ for a general \tef\ $ f $ before giving the proof of Theorem~\ref{Tbasicresults}.

Points in $ J(f) $ whose orbits are dense in $ J(f) $ evidently lie in $ BU(f) $.  Since the set of such points is dense in $ J(f) $ \cite[Lemma~1]{BD2}, it follows that $ BU(f) $ is an infinite, unbounded set. It is known \cite{Bish3, MR918638} that there are \tef s for which $ BU(f) $ meets  $ F(f) $ as well as $ J(f) $. The facts that $ BU(f) $ is completely invariant, and that $ BU(f) = BU(f^n) $ for $ n\isnatural $, follow from the definition of $ BU(f) $ and the corresponding properties for $ I(f) $ and $ K(f) $.  
 
To prove Theorem~\ref{Tbasicresults}, we need the following special case of \cite[Lemma 10]{MR2792984}. For $ f: \C \to \C $ we say that a set $S \subset \C$ is \emph{backwards invariant under $f$} if $z \in S$ implies that $f^{-1}(\{z\}) \subset S$.
\begin{lemma}
\label{Jlemm}
Suppose that $f$ is a {\tef}, and that $E~\subset~\mathbb{C}$ contains at least three points. Suppose also that $E$ is backwards invariant under $f$, that $\operatorname{int} E\cap J(f) = \emptyset$, and that every component of $F(f)$ that meets $E$ is contained in $E$. Then $\partial E = J(f)$.
\end{lemma}
\begin{proof}[Proof of Theorem~\ref{Tbasicresults}]
For part (a), let $U$ be a Fatou component of~$f$ such that $ U \cap BU(f) \neq \emptyset $ and suppose that, if possible, $ z_0 \in U $ does not lie in $ BU(f) $.  Then $ z_0 \in I(f) \cup K(f) $, and it follows by normality that if $ z_0 \in I(f) $ then $ U \subset I(f), $ whilst if $ z_0 \in K(f) $ then $ U \subset K(f). $  Hence there is no such $ z_0 $, which establishes that $U \subset BU(f)$. The fact that $U$ is a wandering domain follows from the classification of periodic Fatou components; see, for example, \cite[Theorem~6]{MR1216719}. \\

For part (b), we apply Lemma~\ref{Jlemm} with $E=BU(f)$.  As discussed above, $BU(f)$ is an infinite set and is backwards invariant. Moreover, the repelling periodic points of $f$ are dense in $J(f)$ \cite[Theorem 1]{Bak68}, and since by definition $BU(f)$ contains no periodic points, it follows that $\operatorname{int} BU(f) \cap J(f) = \emptyset$. Finally, by part (a), every component of $F(f)$ that meets $BU(f)$ is contained in $BU(f)$.  Thus the conditions of Lemma~\ref{Jlemm} are satisfied with $E = BU(f)$, and we conclude that $ J(f) = \partial BU(f) $, as required.
\end{proof}

%
%
%
\section{Connectedness and boundedness properties}
\label{S2}
In this section we prove Theorem~\ref{Tconnectnessandboundedness}, which gives certain connectedness and boundedness properties of $BU(f)$ and related sets. 

Our proof of Theorem~\ref{Tconnectnessandboundedness} requires the use of Corollary~\ref{CboundaryinBU}, which follows from Theorem~\ref{TboundaryinBU}(a), proved in Section~\ref{S3}. We also need a number of other results, which we gather together here in the form of a series of lemmas.  The first, due to Baker \cite{MR759304}, gives some basic properties of \mconn\ Fatou components for \tef s.  We say that the set $ S $ \textit{surrounds} a set or a point if that set or point lies in a bounded complementary component of~$ S $.
\begin{lemma} \label{baker}
Let $ f $ be a \tef\ and let $ U $ be a \mconn\ Fatou component of $ f $.  Then
\begin{itemize}
\item $ f^k(U) $ is bounded for any $ k \in \N $,
\item $ f^{k+1} (U) $ surrounds $ f^k(U) $ for sufficiently large $ k $, and
\item $ \inf \{\vert z \vert: z \in f^k(U)\}  \to \infty $ as $ k \to \infty. $
\end{itemize}
\end{lemma}
Next, we need the following characterisation of a disconnected subset of the plane.
\begin{lemma}\cite[Lemma 3.1]{MR2797684}
\label{Lasse}
Suppose that $S \subset \mathbb{C}$. Then $S$ is disconnected if and only if there is a closed connected set $A \subset \mathbb{C}$ such that $S \cap A = \emptyset$ and at least two different components of $\mathbb{C} \backslash A$ intersect $S$.
\end{lemma}
The third result we need is the well-known `blowing-up' property of the Julia set (see \cite[Section 2]{MR1216719}, for example). Here $E(f)$ is the \emph{exceptional set} of $f$, which consists of the set of points with a finite backwards orbit under $f$. For a {\tef}, $E(f)$ contains at most one point.
\begin{lemma}
\label{Lblow}
Let $f$ be a transcendental entire function, let $K$ be a compact set with $K \cap E(f) = \emptyset$, and let $\Delta$ be a neighbourhood of $z \in J( f)$. Then there exists $N\isnatural$ such that $f^n(\Delta) \supset K$, for $n \geq N$. 
\end{lemma}
Finally we need the following generalisation of \cite[Lemma 1]{MR2792984}, which was proved in \cite{Sixsmithmax}.
\begin{lemma}
\label{RSlemm}
Suppose that $(E_n)_{n\isnatural}$ is a sequence of compact sets and $(m_n)_{n\isnatural}$ is a sequence of positive integers. Suppose also that $f$ is a {\tef} such that $E_{n+1} \subset f^{m_n}(E_n )$, for $n\isnatural$. Set $p_n = \sum_{k=1}^n m_k$, for $n\isnatural$. Then there exists $\zeta\in E_1$ such that 
\begin{equation}
\label{feq}
f^{p_n}(\zeta) \in E_{n+1}, \qfor n\isnatural.
\end{equation}

If, in addition, $E_n \cap J(f) \ne \emptyset$, for $n\isnatural$, then there exists $\zeta \in E_1 \cap J(f)$ such that (\ref{feq}) holds.
\end{lemma}

\begin{proof}[Proof of Theorem~\ref{Tconnectnessandboundedness}]
For part (a), suppose that $f$ is a {\tef} and that $\overline{BU(f)}$ has a bounded component $B$, say. Then there exists a bounded domain~$A$, homeomorphic to an annulus, that surrounds $ B $ and lies in $ BU(f)^c $. Since $BU(f)$ is an infinite set and is completely invariant, it follows from Montel's theorem that $A$ lies in a component of $F(f)$, and since $ B $ meets $ J(f) $ by Theorem~\ref{Tbasicresults}(b), this component of $ F(f) $ must be multiply connected.

On the other hand, if $f$ has a multiply connected Fatou component $U$, then it is immediate from Lemma \ref{baker} that all components of $\overline{BU(f)}$ are bounded. \\

Our proof of part (b) of the theorem uses techniques similar to the proof of \cite[Theorem 1.3]{2012arXiv1208.3692O}. Suppose that $BU(f) \cap J(f)$ is disconnected. It follows from Lemma~\ref{Lasse} that there exists a closed connected set $\Gamma \subset (BU(f) \cap J(f))^c$ with two complementary components, $G_1$ and $G_2$ say, each of which contains points in $BU(f) \cap J(f)$.

For $i = 1,2$, let $H_i$ be a bounded domain, compactly contained in $G_i$, such that $H_i \cap J(f) \ne \emptyset$. Since $J(f)$ is a perfect set, we can assume that neither $\overline{H}_1$ nor $\overline{H}_2$ meets $E(f)$.

Now suppose that $z_1$ is an arbitrary point of $J(f)$, and that $V_1$ is a neighbourhood of $z_1$. Let $(z_n)_{n \geq 2}$ be a sequence of points in $J(f)\backslash E(f)$ such that $z_n\rightarrow\infty$ as $n\rightarrow\infty$. For $ n \geq 2 $, let $ V_n $ be a bounded  neighbourhood of $z_n$.  Evidently we can choose the $ (V_n)_{n \geq 2} $ to be pairwise disjoint and such that $\overline{V}_n \cap E(f) = \emptyset$, for $ n \geq 2 $. Then, by Lemma~\ref{Lblow}, for each $n\isnatural$ there exists $t_n, u_n \isnatural$ such that 
$$
   \overline{H}_1 \cup \overline{H}_2 \subset f^{t_n}(V_n)
$$
and
$$
   \overline{V}_{n+1} \subset f^{u_n}(H_1) \cap f^{u_n}(H_2).
$$

Now let $\underline{s} = s_1 s_2 s_3 \ldots$ be an infinite sequence of $1$s and $2$s. We show that each such sequence $\underline{s}$ can be associated with the orbit of a point in $\overline{V_1} \cap BU(f) \cap J(f)$.  For $n\isnatural$ and $ k = \tfrac12 n$, define
$$
   E_n = \begin{cases}
           &\overline{V}_{k + \frac12}, \quad\quad \, \text{  if  } n \text{ is odd}, \\
           &\overline{H}_{s_k}, \quad\quad\quad \text{ if } n \text{ is even},
         \end{cases}  
$$
and
$$
   m_n = \begin{cases}
           &t_{k + \frac12}, \quad\quad \, \, \text{  if  } n \text{ is odd}, \\
           &u_k, \quad\quad\quad \, \, \text{ if } n \text{ is even}.
         \end{cases}  
$$
Then $E_{n+1} \subset f^{m_n}(E_n )$, for $n\isnatural$, so it follows from Lemma~\ref{RSlemm} that there exists a point $\zeta\in \overline{V}_1 \cap J(f)$, the orbit of which visits $\overline{H}_1 \cup\overline{H}_2$ infinitely often, and also visits each $\overline{V}_n$. Hence $\zeta \in BU(f)$.

Now points in $\overline{V}_1 \cap BU(f) \cap J(f)$ whose orbits are associated with two different infinite sequences of $1$s and $2$s must lie in different components of $BU(f) \cap J(f)$. For if two such sequences differ, then it is easy to see that there exists $m\isnatural$ such that the $m$th iterate of one point lies in $G_1$ and the $m$th iterate of the other lies in~$G_2$. Thus, if the two points were in the same component $B$ of $BU(f) \cap J(f)$, then $f^m(B)$ would meet $\Gamma \subset (BU(f)\cap J(f))^c$, which is a contradiction.

Finally, there are uncountably many possible infinite sequences of $1$s and $2$s, so we have shown that every neighbourhood of an arbitrary point in $J(f)$ meets uncountably many distinct components of $BU(f) \cap J(f)$, as required. \\

For part (c), suppose that $BU(f)$ is disconnected. As in the proof of part (b), it follows from Lemma~\ref{Lasse} that there exists a closed connected set $\Gamma \subset BU(f)^c$ with two complementary components, $G_1$ and $G_2$ say, each of which contains points in~$BU(f)$. 

We claim that each of $G_1$ and $G_2$ must contain points in $J(f)$. For suppose not.  Then we can assume without loss of generality that $G_1 \subset F(f)$. Let $U$ be the Fatou component containing $G_1$. Then $ U \subset BU(f) $ by Theorem~\ref{Tbasicresults}(a), and indeed $ U = G_1 $ since $ \partial G_1 \subset \Gamma \subset BU(f)^c. $ Thus $G_1$ is a component of $BU(f)$ that is also a Fatou component, which is impossible by Corollary~\ref{CboundaryinBU}. This contradiction establishes our claim. 

The remainder of the proof proceeds exactly as for part (b), but we conclude that points in $ \overline{V}_1 \cap BU(f) $ whose orbits are associated with different infinite sequences of 1s and 2s must lie in different components of $ BU(f). $  It then follows that every neighbourhood of an arbitrary point in $ J(f) $ meets uncountably many distinct components of $ BU(f), $ as required.
\end{proof}

\begin{remarks}\normalfont
\begin{enumerate}[(1)]
\item The argument in the proof of Theorem~\ref{Tconnectnessandboundedness}(a) is due to Eremenko~\cite{MR1102727}, who used it to prove that, for a \tef\ $ f $, all components of $ \overline{I(f)} $ are unbounded.  Using the same technique, it is easy to show that, if $ f $ has no \mconn\ Fatou components, then all components of $\overline{K(f)}$ are unbounded, whilst otherwise all components of $\overline{K(f)}$ are bounded.\\
\item Eremenko also conjectured in \cite{MR1102727} that every component of $ I(f) $ is unbounded, and this conjecture remains open.  If $ f $ is a \tef\ with no \mconn\ Fatou components, one can similarly ask whether every component of $ BU(f) $ is unbounded.  We give examples in Section~\ref{S4} to show that this is not the case, and indeed that $ BU(f) $ can be totally disconnected.  Similar remarks apply to $ K(f) $; examples of \tef s for which $ K(f) $ is totally disconnected were given in \cite{2012arXiv1208.3692O}.
\end{enumerate}
\end{remarks}

%
%

%
%

\section{Boundaries of wandering domains}
\label{S3}
In this section we prove Theorem \ref{TboundaryinBU} by first proving Theorem \ref{Tboundaryingeneral}, which is a more general result on the behaviour of iterates of a \tef~$ f $ on the boundary of a wandering domain.  

We first define the term \emph{harmonic measure}, used in the statements of Theorems~\ref{TboundaryinBU} and~\ref{Tboundaryingeneral}. If $U \subset \mathbb{C}$ is a domain, $z \in U$, and $E \subset \partial U$, then the \emph{harmonic measure at $z$ of $E$ relative to $U$}, which we denote by $\omega(z,E,U)$, is the value at $z$ of the solution to the Dirichlet problem in $U$ with boundary values equal to the characteristic function $\chi_E$. We refer to, for example, \cite{MR2150803, MR1334766} for further details. If there exists $z_0 \in U$ such that $\omega(z_0,E,U) = 0$, then $\omega(z,E,U) = 0$ for all $z \in U$. In this case we say that the set $E$ has \emph{harmonic measure zero relative to $U$}. 

We require the following more general result related to \cite[Theorem 1.1]{MR2801622}, which emerged from discussions with Rippon and Stallard. We denote the Riemann sphere by $\sphere := \mathbb{C} \cup \{ \infty \}$, write $d(z,w)$ for the spherical distance between $z, \, w \in \sphere$, and define the spherical distance between a point $z \in \sphere$ and a set $U\subset\sphere$ by $d(z,U)~=~\inf_{w\in U} d(z,w)$. We assume that $d(. , .)$ is normalised so that $d(0, \infty) = 1$.
\begin{lemma}
\label{bdrylemm}
Let $(G_n)_{n\geq 0}$ be a sequence of disjoint simply connected domains in~$ \sphere $. Suppose that, for each $n\isnatural$, $g_n : \overline{G}_{n-1} \to \overline{G}_n$ is analytic in $G_{n-1}$, continuous in $\overline{G}_{n-1}$, and satisfies $g_n(\partial G_{n-1})  \subset \partial G_n$. Let $$h_n = g_n \circ \cdots \circ g_2 \circ g_1, \qfor n\isnatural.$$
Suppose that there exist $\xi\in \sphere$, $\rho \in (0,1)$, $N\isnatural$ and $z_0 \in G_0$ such that 
$$d(h_n(z_0),\xi) < \rho, \qfor n\geq N.$$ 
Suppose finally that $c > 1$, and let 
$$H = \{ z \in \partial G_0 : d(h_n(z),\xi) \geq c\rho \text{ for infinitely many values of } n \}.$$ 
Then $H$ has harmonic measure zero relative to $G_0$.
\end{lemma}
\begin{proof}
By composing with a M\"{o}bius map if necessary, we can assume that $\xi = 0$. Let $r>0$ be such that $d(0,r) = \rho$. By hypothesis we have $|h_n(z_0)| < r$, for $n\geq N$. Let $r' > r$ be such that $d(0, r') = c\rho$. 

Define, for $n\geq N$,
\begin{equation}
\label{BRdef}
   B_n = \{ z \in \partial G_0 : |h_n(z)| \geq r' \}.
\end{equation}
It is easy to see that 
\begin{equation}
\label{Hdef}
H = \bigcap_{m\geq N} \bigcup_{n\geq m} B_n.
\end{equation}

Let $\Delta = \{ z : |z| < r' \}$, and define the following sets, for $n \geq N$:
\begin{itemize}
\item $E_n = \partial G_n \cap \{z : |z| \geq r'\}$;
\item $V_n$ to be the component of $G_n \cap \Delta$ that contains $h_n(z_0)$;
\item $F_n = \partial V_n \cap \{z : |z| = r'\}$.
\end{itemize}

We may assume that none of these sets are empty. By a similar argument to the proof of \cite[Equation (2.2)]{MR2801622} we have
\begin{equation}
\label{newwas2.2}
\omega(z,E_n,G_n) \leq \omega(z,F_n ,\Delta), \qfor z \in V_n, \ n \geq N.
\end{equation}

Since $|h_n(z_0)| < r$ for $n\geq N$, an application of Harnack's inequality \cite[Theorem 1.3.1]{MR1334766} gives
\begin{equation}
\label{newwas2.3}
\omega(h_n(z_0),F_n,\Delta) \leq K \omega(0,F_n ,\Delta), \qfor n \geq N,
\end{equation}
where $K = \dfrac{r'+r}{r' - r}$.  Since the sets $F_n$ are disjoint, it then follows from (\ref{newwas2.2}) and (\ref{newwas2.3}) that 
\begin{equation}
\label{newwas2.4}
\sum_{n \geq N} \omega(h_n(z_0),E_n,G_n) \leq K \sum_{n \geq N} \omega(0,F_n ,\Delta) \leq K \omega(0,\partial\Delta ,\Delta) = K.
\end{equation}

Now $h_n(G_0) \subset G_{n}$ and $h_n(B_n) \subset E_n$, so 
$$
    \omega(z_0, B_n, G_0) \leq \omega(h_n(z_0), E_n, G_n), \qfor n \geq N,
$$
by \cite[Theorem 4.3.8]{MR1334766}.  Thus, by (\ref{newwas2.4}),
$$
    \sum_{n \geq N} \omega(z_0, B_n, G_0) \leq \sum_{n \geq N} \omega(h_n(z_0), E_n, G_n) \leq K,
$$
and therefore
$
    \omega(z_0, B_n, G_0) \rightarrow 0 \text{ as } n\rightarrow\infty.
$ We deduce by (\ref{Hdef}) that $H$ has harmonic measure zero relative to $G_0$, as required.
\end{proof}

\begin{remark}\normalfont
Suppose that, with the hypotheses of Lemma~\ref{bdrylemm}, there exist $\xi \in \sphere$ and $z_0 \in G_0$ such that $h_n(z_0) \rightarrow\xi$ as $n\rightarrow\infty$. It is straightforward to show, by repeated application of the lemma, that the set  $$\{ z \in \partial G_0 : h_n(z) \not\rightarrow\xi \text{ as } n\rightarrow\infty \}$$ has harmonic measure zero with respect to $G_0$. We do not use this result here.
\end{remark}

We now proceed to the proof of Theorem~\ref{Tboundaryingeneral}.  Recall that, for a \tef\ $ f $, we write $\Lambda(z,f)$ for the $\omega$-limit set of the point $ z \in \C$, that is, the set of accumulation points in $\sphere$ of the orbit of $ z $ under $ f $.  For a wandering domain $ U $ of $f$, we can write $\Lambda(U,f)$ without ambiguity, since $\Lambda(z_1,f) = \Lambda(z_2,f)$ for $z_1, z_2 \in U$. Note that $ \Lambda(U,f) \subset J(f) \cup \{ \infty \}. $

\begin{proof}[Proof of Theorem~\ref{Tboundaryingeneral}]
Let $ f $ be a \tef\ and $ U $ be a wandering domain of $ f $. If $U$ is multiply connected, then $\overline{U} \subset I(f)$ by \cite[Theorem 2]{MR2117213}, so the theorem is immediate. We therefore assume in what follows that $U$ is simply connected. 

We show separately that each of the sets
$$\{ z \in \partial U : \Lambda(z,f) \not\subset \Lambda(U, f) \} \quad\text{and}\quad \{ z \in \partial U : \Lambda(z,f) \not\supset \Lambda(U, f) \}$$
is a countable union of sets of harmonic measure zero relative to $ U $. Since a countable union of sets of measure zero also has measure zero, it then follows that the set $ \{ z \in \partial U : \Lambda(z,f) \ne \Lambda(U, f) \} $ has harmonic measure zero relative to $ U $, as required. 

First we consider the set $\{ z \in \partial U : \Lambda(z,f) \not\subset \Lambda(U, f) \}$. For $\delta > 0$, denote by $ A_\delta $ the set of points whose spherical distance from $\Lambda(U,f)$ is at least $\delta$:
$$
   A_\delta := \{ z \in \sphere : d(z,\Lambda(U,f)) \geq \delta \}.
$$
Let $(\delta_p)_{p\isnatural}$ be a decreasing sequence of real numbers in $ (0,1) $, tending to zero, with $\delta_1$ sufficiently small that $A_{3\delta_1} \ne \emptyset$. This is possible since $\sphere\backslash \Lambda(U,f) \supset F(f).$

For each $p\isnatural$, we now choose a sequence $(\zeta_{p,q})_{q\isnatural}$ of points in $A_{3\delta_p}$ such that, with $T_{p,q}$ being the open spherical disc $T_{p,q}:= \{ z \in \sphere : d(z,\zeta_{p,q}) < \delta_p\}$, we have
$$
  A_{3\delta_p} \subset \bigcup_{q\isnatural} \{ z \in \sphere : d(z,\zeta_{p,q}) \leq \delta_p/2 \} \quad\text{and}\quad \bigcup_{q\isnatural} T_{p,q} \subset A_{2\delta_p}.
$$
(Clearly, for fixed $p$, a finite number of these discs is sufficient, but this is not required for our proof.)

For each $p, q\isnatural$, let $H_{p,q}$ be the set of points in the boundary of $U$ with an $\omega$-limit point in $T_{p,q}$. In other words
$$
   H_{p,q} := \{ z \in \partial U : f^n(z) \in T_{p,q}, \text{ for infinitely many values of } n \}.
$$

Fix a point $z_0 \in U$. For each $\delta>0$, we have $f^n(z_0) \notin A_\delta$ for all sufficiently large values of $n$. 

Suppose that $p, q \isnatural$ are fixed. Let $\xi$ be the point opposite to $\zeta_{p,q}$ on the Riemann sphere; in other words $\xi = -1/\overline{\zeta}_{p,q},$ with the obvious modification when $\zeta_{p,q} \in \{0, \infty\}$. Set $\rho = 1 - 2\delta_p$ and $c = (1 - \delta_p)/(1-2\delta_p) > 1$, in which case $T_{p,q} = \{ z : d(z,\xi) > c\rho \}$.  Since the spherical distance of $\zeta_{p,q}$ from $\Lambda(U,f)$ is at least $3\delta_p$, it follows that there exists $N \isnatural$ such that $d(f^n(z_0),\xi) < \rho$, for $n\geq N$. 

Set $G_0 = U$ and, for each $n\isnatural$, let $g_n = f$ and let $G_n$ be the component of $F(f)$ containing $f^n(U)$. Since $ U $ is a simply connected wandering domain, the domains $ (G_n)_{n \geq 0} $ are disjoint and simply connected, and moreover $g_n(\partial G_{n-1})\subset\partial G_{n}$, for $n\isnatural$. Thus we may apply Lemma~\ref{bdrylemm}, and we deduce that $H_{p,q}$ has harmonic measure zero relative to $U$, for all $p, q \isnatural$.

Now suppose that $z \in \partial U$, and that $\zeta \in \Lambda(z,f)$ is such that $\zeta \notin \Lambda(U,f)$. Since $\Lambda(U,f)$ is closed, $d(\zeta, \Lambda(U,f)) > 0$. Hence there exist $p, q\isnatural$ such that $\zeta \in \{ z \in \sphere : d(z,\zeta_{p,q}) \leq \delta_p/2\}$, so that  $f^n(z) \in T_{p,q}$ for infinitely many values of $n$. We deduce that 
$$
\{ z \in \partial U : \Lambda(z,f) \not\subset \Lambda(U, f) \} \subset \bigcup_{p,q \isnatural} H_{p,q}.
$$

Since $H_{p,q}$ has harmonic measure zero relative to $U$ for all $p, q \isnatural$, it follows by countable additivity that $  \{ z \in \partial U : \Lambda(z,f) \not\subset \Lambda(U, f) \} $ also has harmonic measure zero relative to $U$.

Now we consider the set $\{ z \in \partial U : \Lambda(z,f) \not\supset \Lambda(U, f) \}$, and let $(\delta_p)_{p\isnatural}$ be the sequence introduced earlier.  For each $p\isnatural$, we choose a sequence $(\zeta'_{p,q})_{q\isnatural}$ of points in $\Lambda(U,f)$ such that, with $T'_{p,q}$ being the open spherical disc $T'_{p,q} = \{ z \in \sphere : d(z,\zeta'_{p,q}) < \delta_p\}$, we have
$$
  \Lambda(U,f) \subset \bigcup_{q\isnatural} {T'_{p,q}}.
$$
(Again, for fixed $p$, a finite number of these discs is sufficient, but this fact is not required for our proof).

For each $p, q\isnatural$, let $H'_{p,q}$ be the set of points in the boundary of $U$ whose orbits eventually lie outside $T'_{p,q}$. In other words,
$$
   H'_{p,q} := \{ z \in \partial U : f^n(z) \notin T'_{p,q}, \text{ for all sufficiently large values of } n \}.
$$

Suppose that $p, q \isnatural$ are fixed. Since $\zeta'_{p,q}\in \Lambda(U,f)$, there exist $w \in U$ and an increasing sequence of positive integers $(q_n)_{n\isnatural}$ such that 
$$
f^{q_n}(w) \in \{ z \in \sphere : d(z,\zeta'_{p,q}) < \delta_p/2\}, \qfor n\isnatural.
$$ 

We again apply Lemma~\ref{bdrylemm}. Let $g_1 = f^{q_1}$ and, for $n \geq 2$, set $g_n = f^{q_n - q_{n-1}}$, so that
$$
g_n \circ \cdots \circ g_2 \circ g_1 = f^{q_n}, \qfor n\isnatural.
$$
Set $G_0 = U$ and, for each $n\isnatural$, let $G_n$ be the component of $F(f)$ containing $f^{q_n}(U)$.  Also set $z_0 = w$, $\xi = \zeta'_{p,q}$, $\rho = \delta_p/2$ and $c = 2$.  Then it follows from Lemma~\ref{bdrylemm} that $H'_{p,q}$ has harmonic measure zero relative to $U$, for all $p, q \isnatural$.

Now suppose that $z \in \partial U$, and that $\zeta' \in \Lambda(U,f)$ is such that $\zeta' \notin \Lambda(z,f)$. Since $\zeta' \notin \Lambda(z,f)$, there exists $\delta>0$ such that $d(f^n(z),\zeta') > \delta$ for all sufficiently large values of $n$. Thus, since $\delta_p\rightarrow 0$ as $p\rightarrow\infty$, there exist $p, q\isnatural$ such that $f^n(z) \notin T'_{p,q}$, for all sufficiently large values of $n$. We deduce that 
$$
\{ z \in \partial U : \Lambda(z,f) \not\supset \Lambda(U, f) \} \subset \bigcup_{p,q \isnatural} H'_{p,q},
$$
which has harmonic measure zero relative to $U$, by countable additivity.  This completes the proof.
\end{proof}

We next show that Theorem~\ref{TboundaryinBU} is a straightforward consequence of Theorem~\ref{Tboundaryingeneral}.
\begin{proof}[Proof of Theorem~\ref{TboundaryinBU}]
For part (a), suppose that $U \subset BU(f)$. Then $U$ is a wandering domain by Theorem~\ref{Tbasicresults}(a). Since $U \subset BU(f)$, there exists $\zeta\in\mathbb{C}$ such that $\{\zeta, \infty\} \subset \Lambda(U,f)$.  Hence it follows from Theorem~\ref{Tboundaryingeneral} that the set $$\{ z \in \partial U : \{\zeta, \infty\} \not\subset \Lambda(z,f) \}$$ has harmonic measure zero relative to $U$ and therefore, in particular, so does the set $\{ z \in \partial U : z \in BU(f)^c \}.$

For part (b) we prove the contrapositive. Suppose that $U$ is a wandering domain and that $U \not \subset BU(f)$. Then $ U \cap BU(f) = \emptyset $ by Theorem~\ref{Tbasicresults}(a), so $ U \subset I(f) $ or $ U \subset K(f). $  Now if $U \subset I(f)$, we have $\Lambda(U,f) = \{\infty\}$, so it follows from Theorem~\ref{Tboundaryingeneral} that $\Lambda(z,f) = \{\infty\}$ for all $z \in\partial U$ except at most a set of harmonic measure zero relative to $U$. On the other hand, if $U \subset K(f)$, then $\infty \notin \Lambda(U,f)$, and it follows from Theorem~\ref{Tboundaryingeneral} that $\infty \notin \Lambda(z,f)$, for all $z \in\partial U$ except at most a set of harmonic measure zero relative to $U$. In either case we deduce that $\partial U \cap BU(f)$ has harmonic measure zero relative to $U$, and the result follows.
\end{proof}

\begin{remarks}\normalfont
\begin{enumerate}[(1)]
\item Suppose that $f$ is a {\tef}, and that $U$ is a wandering domain of $f$. Using a similar argument to the proof of Theorem~\ref{TboundaryinBU}, it is easy to see that if $U\subset K(f)$, then $\partial U \cap K(f)^c$ has harmonic measure zero relative to $U$, and, in the other direction, that if $\partial U \cap K(f)$ has positive harmonic measure relative to $U$, then $U \subset K(f)$. We omit the details.

However, it is unknown whether wandering domains in $K(f)$ can exist for \tef s.  In \cite[Theorem 1.4]{2012arXiv1208.3692O}, the first author showed that such functions with the property of being \emph{strongly polynomial-like} have no wandering domains in $ K(f) $; we refer to \cite{2012arXiv1208.3692O} for the definition.\\
\item Our proof of Theorem~\ref{Tboundaryingeneral} allows for the possibility that the $\omega$-limit set $ \Lambda(U,f) $ of a wandering domain $ U $ of a \tef\ $ f $ could be an uncountable set.  We are unaware of any examples where this is the case, and it is an interesting question whether or not it can occur.  Our proof of Theorem~\ref{Tboundaryingeneral} could be simplified if it cannot. We  observe that Rempe-Gillen and Rippon \cite{exoticbaker} showed that there are holomorphic functions between \emph{Riemann surfaces} whose $\omega$-limit set is uncountable. However, in their examples the limit set is contained in the boundary of the Riemann surface, which in our context is a single point.
\end{enumerate}
\end{remarks}
%
%
%
%
%
%
%
%
\section{Hausdorff dimension of $BU(f)$}
\label{Sdimension}
In this section we show that a recent result in \cite{RempeUrbanski} enables us to draw conclusions about the Hausdorff dimension of $BU(f) $ that are similar to known results for $ K(f) $.  We denote the Hausdorff dimension of a set~$A$ by $\dim_H A$, and refer to \cite{falconer} for a definition and further information. 

Stallard \cite{MR1260113} used the Ahlfors five islands theorem to show that $\dim_H K(f) > 0$. A stronger result holds for functions in the \emph{Eremenko-Lyubich class}~$\mathcal{B}$, that is, those {\tef}s with a bounded set of singular values (a \emph{singular value} is, by definition, a critical value or a finite asymptotic value). For functions in this class, Stallard \cite{MR1357062} showed that $\dim_H J(f) > 1$, whilst Bara{\'n}ski, Karpi{\'n}ska and Zdunik \cite{MR2480096} showed that $\dim_H (K(f) \cap J(f))>1$.

For $BU(f)$ we have the following analogous result.
\begin{theorem}
\label{Tdim}
Suppose that $f$ is a {\tef}. Then:
\begin{enumerate}[(a)]
\item $\dim_H (BU(f) \cap J(f)) > 0.$\label{Tdima}
\item If, in addition, $f \in \mathcal{B}$, then $\dim_H (BU(f) \cap J(f)) > 1.$ \label{Tdimb}
\end{enumerate}
\end{theorem}

We need the idea of the \emph{hyperbolic dimension} of a \tef\ $ f $.  This is defined to be the supremum of the Hausdorff dimensions of hyperbolic subsets of $ J(f) $, where $ K \subset J(f) $ is \emph{hyperbolic} if it is compact and forward invariant, and if sufficiently large iterates of $ f $ are expanding when restricted to $ K $.  We refer to \cite{RempeUrbanski} and references therein for more details.
  
\begin{proof}[Proof of Theorem \ref{Tdim}]
Suppose that $f$ is a {\tef}. Let $J_{d}(f)$ denote the set of points in $J(f)$ whose orbits are dense in $J(f)$. It is shown in \cite[Theorem 1.4]{RempeUrbanski} that the Hausdorff dimension of $J_{d}(f)$ is greater than or equal to the hyperbolic dimension of $f$. The proof of Stallard's result \cite{MR1357062} in fact shows that the Hausdorff dimension of a hyperbolic subset of $ K(f) \cap J(f) $ is greater than zero, so it follows that the hyperbolic dimension of $f$ is greater than zero. Since  $J_{d}(f) \subset BU(f) \cap J(f)$, this proves part (a).

Similarly, the proof in \cite{MR2480096} shows that, for $f \in \mathcal{B}$, the Hausdorff dimension of a hyperbolic subset of $ K(f) \cap J(f) $ is greater than one. Hence part (\ref{Tdimb}) of the theorem follows in the same way as part (\ref{Tdima}).
\end{proof}
\begin{remarks}\normalfont
\begin{enumerate}[(1)]
\item Bishop's construction in \cite{Bish2} shows that Theorem~\ref{Tdim}(a) is sharp. He shows that, given $ \alpha > 0$, there exists a {\tef} $f$ such that all Fatou components of $ f $ lie in a subset $A(f)$ of $I(f)$ known as the fast escaping set, and $\dim_H J(f) \backslash A(f) < \alpha$. \\
\item The above results on the Hausdorff dimension of $K(f) \cap J(f)$ and $BU(f) \cap J(f)$ stand in contrast to the situation for $I(f) \cap J(f)$. For a general \tef\ $ f $, it is known \cite{RS} that $I(f) \cap J(f)$ contains nondegenerate continua, so 
$
\dim_H (I(f) \cap J(f)) \geq 1.
$
Moreover, there is a function $f \in \mathcal{B}$ such that $\dim_H I(f) = 1$; see \cite[Theorem 1.1]{RGS}.\\
\item We are grateful to Lasse Rempe-Gillen for drawing our attention to Theorem~\ref{Tdim}(b), which strengthens a result we had proved earlier using the methods of \cite{MR3054344}. We also note that it is possible to prove Theorem~\ref{Tdim}(a) directly using the same technique as in \cite{MR1260113}.
\end{enumerate}
\end{remarks}
%
%
%
%
%
%
%
\section{Examples}
\label{S4}

In this section, we illustrate some of the different topological and dynamical structures that can occur for $ BU(f) $.  First, we give an example of a \tef\ for which $ BU(f) $ is totally disconnected.  

\begin{example}
\label{exampledisconn}
Let
\[ f(z) = \lambda e^z, \qfor \lambda \in (0, 1/e). \]
Then $BU(f)$ is totally disconnected.
\end{example}

\begin{proof}
It was shown by Devaney and Tangerman \cite{MR873428} that $F(f)$ is a completely invariant immediate attracting basin, and that $J(f)$ is a \emph{Cantor bouquet}, which consists of uncountably many disjoint curves, each with one finite endpoint and the other endpoint at infinity. It is well known that these curves, apart from some of the finite endpoints, lie in $I(f)$ (see \cite{MR1956142} and references therein). Mayer \cite{MR1053806} showed that the set of finite endpoints is totally disconnected. Since $ BU(f) $ is a subset of this set of endpoints, it follows that $BU(f)$ is totally disconnected.
\end{proof}

\begin{remarks}\normalfont
\begin{enumerate}[(1)]
\item The function $ f $ in Example \ref{exampledisconn} is in the Eremenko-Lyubich class $ \mathcal{B} $ \---\ indeed, it has only one singular value. It can be shown by an argument similar to that in \cite[Example 5.4]{2012arXiv1208.3692O} that $ BU(f) $ is also totally disconnected for the function $ f(z) = z + 1 + e^{-z} $, first studied by Fatou, which is not in the class~$ \mathcal{B} $.  We omit the details.\\
\item More generally, there are many \tef s for which $ BU(f) $ consists of uncountably many bounded components, though all components of $\overline{BU(f)}$ are unbounded.  In particular, this is the case whenever $ f $ has no \mconn\ Fatou components (recall Theorem~\ref{Tconnectnessandboundedness}(a)) and $ I(f) $ takes the form of a \sw\ \---\ we refer to \cite{Rippon01102012, DS1} for a discussion of the terminology used here and examples of functions of this type.  For some such functions, for instance the functions of small growth referred to in \cite[Examples 1 and 2]{pre05533139}, all components of $ BU(f) $ also have empty interior.
\end{enumerate}
\end{remarks}

%

Next, we show that there is a \tef\ for which $ BU(f) $ has uncountably many unbounded components.

\begin{example}
\label{exampleunc}
Let
\[ f(z) = \lambda e^z, \qfor \lambda > 1/e. \]
Then $BU(f)$ has uncountably many unbounded components, each with empty interior.
\end{example}

\begin{proof}
Since $J(f) = \mathbb{C}$ (see the survey article \cite{Dev2} and references therein), it is clear that all components of $BU(f)$ have empty interior. It remains to show that $BU(f)$ has uncountably many unbounded components.

Our argument is based on a construction due to Devaney \cite{MR1257026} and its extension in \cite{DevaneyJarque}.  We first describe Devaney's construction and use it to show that $ BU(f) $ has an unbounded component, and then indicate how the argument in \cite{DevaneyJarque} implies the existence of uncountably many such components.  We adopt the same notation as in \cite{MR1257026} for ease of reference. 

Devaney's construction begins with the closed horizontal strip, $$S = \{z : 0 \leq \operatorname{Im } z \leq \pi\}.$$
Since $f(\partial S) \subset \mathbb{R}$, it follows that $\partial S$ and all its preimages lie in $I(f)$.  Note that $f$ maps the interior of $S$ onto the upper half-plane, so some points in $ S $ are mapped outside $S$, while others are mapped inside $S$. 

Denote by $ \Lambda $ the set of points whose whole orbit lies inside $S$,
 $$\Lambda := \{ z \in S : f^n(z) \in S, \text{ for } n \isnatural\},$$
and by $L_n$ the set of points in $S$ that leave $S$ at the $n$th iteration of $f$,
$$
  L_n := \{z : f^k(z) \in S, \text{ for } 0 \leq k < n, \text{ and } f^n(z) \notin S\}, \qfor n\isnatural.
$$
Finally, let $B_n = \partial L_n$, for $n\isnatural$. It is easy to see that $B_n \subset \Lambda \cap I(f)$, for $n\isnatural$.  The reader may find it helpful to refer to the pictures of these sets given in \cite[Figure 3]{MR1257026} or \cite[Figure 1]{MR2599894}. 

%
%

For each $n\isnatural$, define $J_n = \bigcup_{k=n}^\infty B_k$. Devaney proves that each $J_n$ is dense in~$\Lambda$, and uses this fact to show that $\Lambda$ is connected.  He also shows \cite[p.632]{MR1257026} that, if $z \in \Lambda$, then exactly one of the following applies:
\begin{enumerate}
\item $z$ is a unique fixed point $p_\lambda$;
\item $z\in B_n$, for some $n\isnatural$, so $z \in I(f)$;
\item the $\omega$-limit set of $z$ is the orbit of zero plus the point at infinity.
\end{enumerate}

Next, Devaney compactifies $\Lambda$ by adjoining the backward orbit of zero. Roughly speaking, this is achieved by adding `points at infinity' that `join together' the ends of the curves $B_n$. The union of $\Lambda$ and these `points at infinity' is denoted by $\Gamma$, which is compact and connected and hence a \emph{continuum}.  Moreover, Devaney uses a result of Curry \cite[Theorem 8]{MR1108604} to deduce that $ \Gamma $ is an \emph{indecomposable} continuum \---\ that is, a continuum that cannot be written as the union of two proper subcontinua. We refer to \cite{MR1257026} for the full details. 

Now if $Z$ is a nondegenerate continuum and $x \in Z$, the \emph{composant} of $Z$ containing $x$ is defined by 
$$
   \{ y \in Z : \text{there is a proper subcontinuum } C \subset Z \text{ such that } x, y \in C \}.
$$
It is known \cite[Theorem 11.15 and Theorem 11.17]{MR1192552} that a nondegenerate indecomposable continuum has uncountably many composants and that these are pairwise disjoint. 

We consider the composants of $\Gamma$. If a composant $X$ of $\Gamma$ meets $B_n$ for some $n\isnatural$, then $B_n \subset X$ since $ B_n $ is an unbounded Jordan curve. Thus there are at most countably many composants of $ \Gamma $ that either intersect $B_n$ for some $n\isnatural$, or contain $p_\lambda$. Any other composant of $ \Gamma $, of which there are uncountably many, lies in $ BU(f) $ by (3) above. 

Now each composant of a continuum is dense in that continuum; see, for example, \cite[5.20]{MR1192552}. Thus every composant of $ \Gamma $ is unbounded, and we deduce that $BU(f)$ has at least one unbounded component.

In \cite{DevaneyJarque}, Devaney and Jarque strengthened the original result in \cite{MR1257026} by showing that there are, in fact, uncountably many indecomposable continua with properties similar to $\Gamma$; see, in particular, \cite[Theorem 5.1 and Theorem 6.1]{DevaneyJarque}. Using an argument similar to that above, it is easy to see that each of these indecomposable continua meets an unbounded component of $BU(f)$.

Suppose that $\Gamma_1 \neq \Gamma_2$ are two such indecomposable continua, meeting the unbounded components $ \beta_1 $ and $ \beta_2 $ of $BU(f)$ respectively. Now it follows from Devaney and Jarque's construction that there exists $n\isnatural$ such that $f^n(\Gamma_1)$ and $f^n(\Gamma_2)$ lie in different horizontal strips of the form $$\{ z : (2p-1)\pi < \operatorname{Im}(z)\leq (2p+1)\pi\}, \qfor p\in\mathbb{Z}.$$ Since the boundaries of these strips lie in $I(f)$, we deduce that $\beta_1 \ne \beta_2$, and hence that $BU(f)$ has uncountably many unbounded components, as claimed.
\end{proof}
\begin{remark}\normalfont
By a similar argument, it follows from \cite[Theorem 1.2]{Nonlanding} that $BU(f) $ has uncountably many unbounded components with empty interior for any function $ f $ in the exponential family whose singular value is on a dynamic ray or is the landing point of such a ray. We refer to \cite{Nonlanding} for definitions.
\end{remark}

Finally, we show that $ BU(f) $ can contain unbounded Fatou components.

\begin{example}
\label{exampleunbdfatou}
There exists a {\tef} $f$ such that $BU(f)$ contains an unbounded Fatou component.
\end{example}

\begin{proof}
We outline how Bishop's example of a \tef\ in the Eremenko-Lyubich class $ \mathcal{B} $ which has a wandering domain in $BU(f)$ \cite[Section 17]{Bish3} can easily be modified to give a function which also has an unbounded Fatou component in~$ BU(f) $.

For the convenience of the reader, we first give a brief description of Bishop's construction. For full details and definitions of terminology, we refer to \cite{Bish3}. Also useful is \cite[Section 3]{FGJ}, which discusses Bishop's example in depth and verifies certain points that are left to the reader in \cite{Bish3}. In particular, it follows from results in \cite{FGJ} that all the Fatou components in Bishop's example are bounded.

The paper \cite{Bish3} introduces \emph{quasiconformal folding}, a new technique for constructing \tef s with good control over geometry and singular values.  Starting from an infinite connected graph that satisfies certain geometric conditions, Bishop shows how to combine carefully chosen quasiconformal maps on the complementary components of the graph into a map that is continuous across the graph and quasiregular on the whole plane.  An entire function with similar properties to the quasiregular map is then obtained by the measurable Riemann mapping theorem.

The key result used in Bishop's example is \cite[Theorem 7.2]{Bish3}. Here, the complex plane is divided by a graph into domains known as \emph{R-components}, \emph{L-components} and \emph{D-components}, with certain quasiconformal maps defined in each. Subject to some technical constraints, for which we refer to \cite{Bish3}, these components and quasiconformal maps are as follows.  We denote the unit disc $B(0,1)$ by $\mathbb{D}$.
\begin{enumerate}[(1)]
\item R-components are unbounded. The required quasiconformal map on an R-component is the composition of a quasiconformal map to the right half-plane and another map, which in this case we can take to be $z \mapsto \cosh(z)$.  Note that R-components are the only components on which Bishop's new technique of quasiconformal folding is needed.
\item L-components are also unbounded, and share edges only with R-components. The required quasiconformal map on an L-component is the composition of a quasiconformal map to the left half-plane, the exponential map to $\mathbb{D}\backslash\{0\}$ and (optionally) a quasiconformal map from $\mathbb{D}$ to $\mathbb{D}$ that takes the origin to another point in $\mathbb{D}$. 
\item D-components are bounded, and share edges only with R-components. The required quasiconformal map on a D-component is the composition of a quasiconformal map to $\mathbb{D}$, a power map $z \to z^d$ and (optionally) a quasiconformal map from $\mathbb{D}$ to $\mathbb{D}$ that takes the origin to another point in $\mathbb{D}$. 
\end{enumerate}

Note that a graph consisting only of R-components is a tree, and the corresponding \tef\ has only two singular values, namely critical values at $ \pm 1. $  Adding D-components to the graph enables the introduction of critical points of any degree, and adding L-components enables the inclusion of finite asymptotic values.
  


The graph used in Bishop's example is symmetrical about both the real and imaginary axes, and does not use L-components. The construction is very delicate, in that the properties of the D-components depend on the function resulting from \cite[Theorem 7.2]{Bish3}. There are then further \emph{post hoc} adjustments to the construction. A key aspect of Bishop's argument is that any apparent circularity here can be controlled. We do not attempt to discuss this detail.

One R-component is the strip $$S_+ = \{ z = x + iy : x > 0, |y| < \pi / 2 \}.$$ The quasiconformal map in $S_+$ is the (in fact, conformal) map $z \mapsto \cosh(\lambda\sinh(z))$, where $\lambda \in \pi \N$ is chosen sufficiently large that the point $\frac{1}{2}$ tends to infinity along the real axis under iteration by $f$.

\begin{figure}[ht]
	\centering
	\includegraphics[width=12cm,height=9cm]{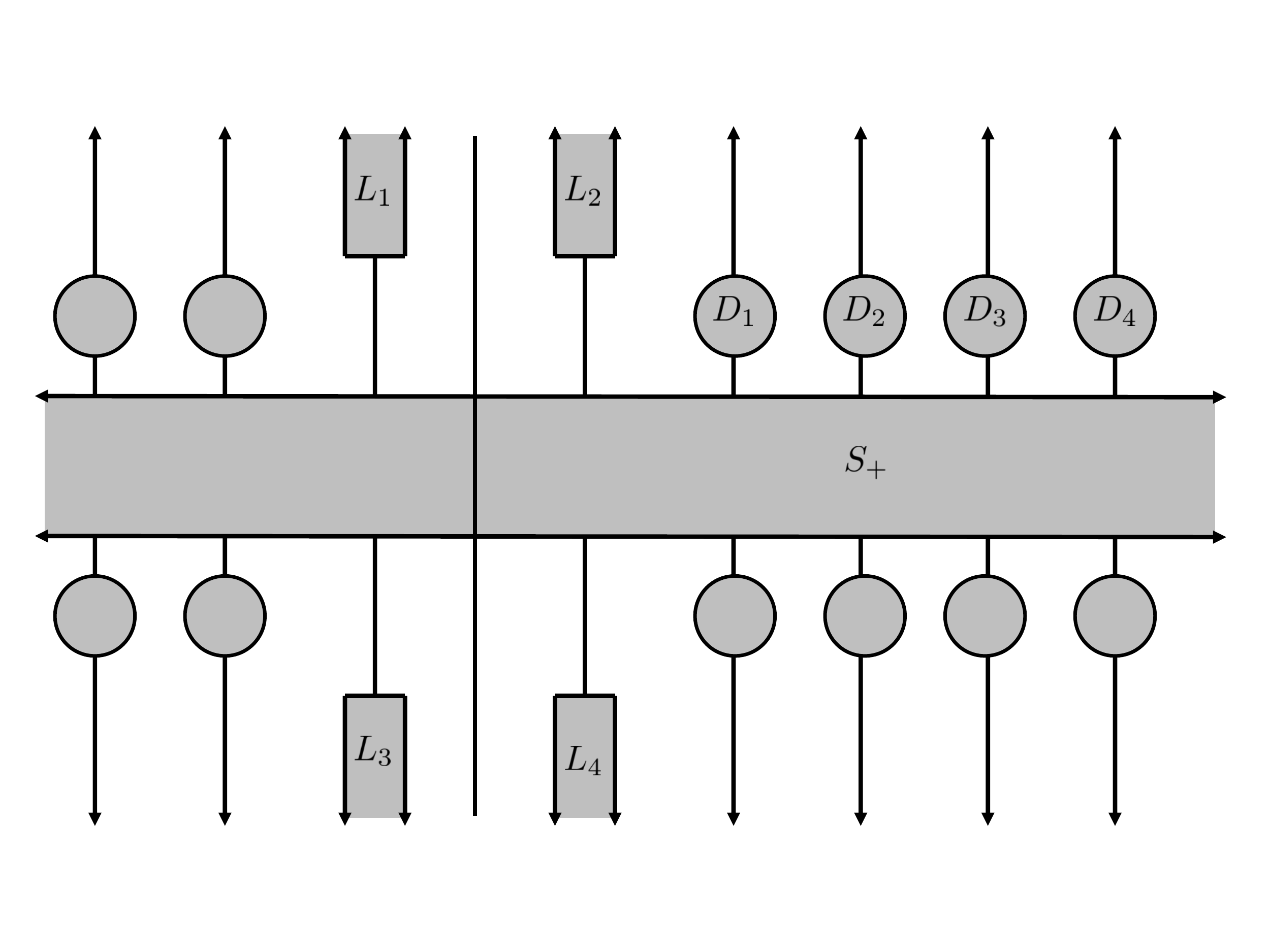}
	\caption{Sketch of the graph for our modification of Bishop's construction (replacing Bishop's Figure~41), showing the placement of the four additional L-components. The graph is not drawn to scale.}
  \label{fig41}
\end{figure}

The D-components are disjoint discs of unit radius, centred at points of imaginary part $\pm \pi$. The quasiconformal maps on these D-components are compositions of a translation to $\mathbb{D}$, a power map of high degree and a quasiconformal map that takes the origin to a point close to $\frac{1}{2}$. The positioning of the D-components, the degree of the power map and the choice of the point close to $\frac{1}{2}$ are all carefully controlled.

The remainder of the complex plane is divided into R-components, but since the dynamics in these components does not affect the example, the quasiconformal maps are not specified.

Choosing a small domain $U$ in $S_+$ close to $\frac{1}{2}$ and with positive imaginary part, it is shown that the iterates of $ U $ under $ f $ follow the orbit of $\frac{1}{2}$ until -- through careful choice of the location of the D-components -- the $n$th iterate (say) lands in a D-component. The quasiconformal map in this D-component is selected so as to reduce the diameter of $f^n(U)$ by a large factor (by using a power map of sufficiently high degree), and return it even closer to $\frac{1}{2}$. Subsequent iterates again follow the orbit of~$\frac{1}{2}$ but, because they start closer to this point, they do so for longer before landing in a D-component further from the origin.  Bishop shows that $U \subset F(f)$, and indeed that $ U $ is a wandering domain. It is easy to see that $U \subset BU(f)$.

Our only change to Bishop's construction is to add some extra components to his graph, and fix the behaviour of the function in these components. We do this in such a way that we add a further property to the dynamics of the function, without disturbing Bishop's construction of a wandering domain in $BU(f)$. 

Specifically we add four L-components to Bishop's graph \---\ see Figure~\ref{fig41} and compare \cite[Figure 41]{Bish3} (we use four L-components to preserve the symmetry of the graph). Note that this introduces four additional R-components, but these do not affect the construction.

In the L-component in the first quadrant, labeled $L_2$ in Figure~\ref{fig41}, we define the required quasiconformal map as the composition of a map to a left half plane, the exponential map to $\mathbb{D}\backslash\{0\}$ and a quasiconformal map from $\mathbb{D}$ to $\mathbb{D}$ that maps the origin to a point in the domain $U$ defined above. The rest of Bishop's construction is then followed without further amendment.

Since $U$ is a domain in $F(f) \cap BU(f)$ and since, by construction, $f$ has a finite asymptotic value in $U$, we deduce that $ f $  has an unbounded Fatou component in $BU(f)$, as claimed.
\end{proof}
%
%
%
%
%
%
\emph{Acknowledgment:} The authors are grateful to Gwyneth Stallard and Phil Rippon for their help with this paper, and also to Lasse Rempe-Gillen for suggesting Example~\ref{exampleunc} and a number of other helpful comments. Finally we are grateful to Anne Sixsmith for proposing the informal term \emph{bungee set} for $BU(f)$.
\bibliographystyle{acm}

\begin{thebibliography}{10}

\bibitem{Bak68}
{\sc Baker, I.~N.}
\newblock Repulsive fixpoints of entire functions.
\newblock {\em Math. Z. 104}, (1968), 252--256.

\bibitem{MR759304}
{\sc Baker, I.~N.}
\newblock Wandering domains in the iteration of entire functions.
\newblock {\em Proc. London Math. Soc. (3) 49}, 3 (1984), 563--576.

\bibitem{BD2}
{\sc Baker, I.~N. and Dom\'{i}nguez, P.}
\newblock Residual Julia sets.
\newblock {\em J. Analysis 8}, (2000), 121--137.

\bibitem{MR2480096}
{\sc Bara{\'n}ski, K., Karpi{\'n}ska, B., and Zdunik, A.}
\newblock Hyperbolic dimension of {J}ulia sets of meromorphic maps with
  logarithmic tracts.
\newblock {\em Int. Math. Res. Not. 2009}, 4 (2009),
  615--624.

\bibitem{aB}
{\sc Beardon, A.F.}
\newblock {\em Iteration of rational functions},
\newblock Graduate Texts in Mathematics 132. Springer-Verlag, 1991.

\bibitem{MR1216719}
{\sc Bergweiler, W.}
\newblock Iteration of meromorphic functions.
\newblock {\em Bull. Amer. Math. Soc. (N.S.) 29}, 2 (1993), 151--188.

\bibitem{MR2869069}
{\sc Bergweiler, W.}
\newblock On the set where the iterates of an entire function are bounded.
\newblock {\em Proc. Amer. Math. Soc. 140}, 3 (2012), 847--853.

\bibitem{MR3054344}
{\sc Bergweiler, W., and Peter, J.}
\newblock Escape rate and {H}ausdorff measure for entire functions.
\newblock {\em Math. Z. 274}, 1-2 (2013), 551--572.

\bibitem{Bish3}
{\sc Bishop, C.~J.}
\newblock Constructing entire functions by quasiconformal folding.
\newblock {\em Acta Math. 214}, 1 (2015), 1--60.

\bibitem{Bish2}
{\sc Bishop, C.~J.}
\newblock A transcendental {J}ulia set of dimension $1$.  \textit{Preprint.}\\
\newblock {\em http://www.math.sunysb.edu/$~\sim$bishop/papers.\/}


\bibitem{Cremer}
{\sc Cremer, H.}
\newblock \"{U}ber die {S}chr\"{o}dersche {F}unktionalgleichung und das
  {S}chwartsche {E}ckenabbildungsproblem.
\newblock {\em Ber. Verh. Sachs. Akad. Wiss. Leipzig, Math.-Phys. Kl. 84\/}
  (1932).

\bibitem{MR1108604}
{\sc Curry, S.~B.}
\newblock One-dimensional nonseparating plane continua with disjoint
  {$\epsilon$}-dense subcontinua.
\newblock {\em Topology Appl. 39}, 2 (1991), 145--151.

\bibitem{MR1257026}
{\sc Devaney, R.~L.}
\newblock Knaster-like continua and complex dynamics.
\newblock {\em Ergodic Theory Dynam. Systems 13}, 4 (1993), 627--634.

\bibitem{Dev2}
{\sc Devaney, R.~L.}
\newblock Cantor bouquets, explosions, and Knaster continua: dynamics of complex exponentials.
\newblock {\em Publ. Mat.}, 43 (1) (1999), 27--54.

\bibitem{DevaneyJarque}
{\sc Devaney, R.~L., and Jarque, X.}
\newblock Indecomposable continua in exponential dynamics.
\newblock {\em Conform. Geom. Dyn. 6\/} (2002), 1--12.

\bibitem{MR873428}
{\sc Devaney, R.~L., and Tangerman, F.}
\newblock Dynamics of entire functions near the essential singularity.
\newblock {\em Ergodic Theory Dynam. Systems 6}, 4 (1986), 489--503.

\bibitem{MR1102727}
{\sc Eremenko, A.~E.}
\newblock On the iteration of entire functions.
\newblock {\em Dynamical systems and ergodic theory ({W}arsaw 1986) 23\/}
  (1989), 339--345.

\bibitem{MR918638}
{\sc Eremenko, A.~E., and Lyubich, M.~Y.}
\newblock Examples of entire functions with pathological dynamics.
\newblock {\em J. Lond. Math. Soc. (2) 36}, 3 (1987), 458--468.

\bibitem{MR1196102}
{\sc Eremenko, A.~E., and Lyubich, M.~Y.}
\newblock Dynamical properties of some classes of entire functions.
\newblock {\em Ann. Inst. Fourier (Grenoble) 42}, 4 (1992), 989--1020.

\bibitem{FGJ}
{\sc {Fagella}, N., {Godillon}, S., and {Jarque}, X.}
\newblock {Wandering domains for composition of entire functions}.
\newblock {\em Preprint, arXiv:1410.3221v1\/} (2014).

\bibitem{falconer}
{\sc Falconer, K.}
\newblock {\em Fractal geometry: mathematical foundations and applications},
  second~ed.
\newblock Wiley, 2006.

\bibitem{MR1504797}
{\sc Fatou, P.}
\newblock Sur les \'equations fonctionnelles.
\newblock {\em Bull. Soc. Math. France 48\/} (1920), 33--94.

\bibitem{MR2150803}
{\sc Garnett, J.~B., and Marshall, D.~E.}
\newblock {\em Harmonic measure}, vol.~2 of {\em New Mathematical Monographs}.
\newblock Cambridge University Press, Cambridge, 2005.

\bibitem{MR1053806}
{\sc Mayer, J.~C.}
\newblock An explosion point for the set of endpoints of the {J}ulia set of
  {$\lambda\exp(z)$}.
\newblock {\em Ergodic Theory Dynam. Systems 10}, 1 (1990), 177--183.

\bibitem{Mil}
{\sc Milnor, J.}
\newblock {\em Dynamics in one complex variable},
  Third Edition.
\newblock Princeton University Press, 2006.

\bibitem{MR1192552}
{\sc Nadler, Jr., S.~B.}
\newblock {\em Continuum theory}, vol.~158 of {\em Monographs and Textbooks in
  Pure and Applied Mathematics}.
\newblock Marcel Dekker, Inc., New York, 1992.

\bibitem{2012arXiv1208.3692O}
{\sc Osborne, J.~W.}
\newblock {Connectedness properties of the set where the iterates of an entire
  function are bounded}.
\newblock {\em Math. Proc. Cambridge Philos. Soc.}, 155(3):391--410, 2013.

\bibitem{Kfc}
{\sc Osborne, J.~W., Rippon, P.~J., and Stallard, G.~M.}
\newblock {Connectedness properties of the set where the iterates of an entire
  function are unbounded}.
\newblock {\em In preparation\/}.

\bibitem{MR1334766}
{\sc Ransford, T.}
\newblock {\em Potential theory in the complex plane}, vol.~28 of {\em London
  Mathematical Society Student Texts}.
\newblock Cambridge University Press, Cambridge, 1995.

\bibitem{Nonlanding}
{\sc Rempe, L.}
\newblock On nonlanding dynamic rays of exponential maps.
\newblock {\em Ann. Acad. Sci. Fenn. Math. 32}, 2 (2007), 353--369.

\bibitem{MR2599894}
{\sc Rempe, L.}
\newblock The escaping set of the exponential.
\newblock {\em Ergodic Theory Dynam. Systems 30}, 2 (2010), 595--599.

\bibitem{MR2797684}
{\sc Rempe, L.}
\newblock Connected escaping sets of exponential maps.
\newblock {\em Ann. Acad. Sci. Fenn. Math. 36}, 1 (2011), 71--80.

\bibitem{exoticbaker}
{\sc Rempe, L., and Rippon, P.~J.}
\newblock Exotic {B}aker and wandering domains for {A}hlfors islands maps.
\newblock {\em J. Anal. Math. 117\/} (2012), 297--319.

\bibitem{RGS}
{\sc Rempe, L., and Stallard, G.~M.}
\newblock Hausdorff dimensions of escaping sets of transcendental entire functions.
\newblock {\em Proc. Amer. Math. Soc. 138\/} (2010), 1657--1665.

\bibitem{RempeUrbanski}
{\sc Rempe-Gillen, L., and Urba{\'n}ski, M.}
\newblock Non-autonomous conformal iterated function systems and {M}oran-set
  constructions.
\newblock {\em To appear in Trans. Amer. Math. Soc., arXiv:1210.7469v4\/}
  (2014).

\bibitem{MR2117213}
{\sc Rippon, P.~J., and Stallard, G.~M.}
\newblock On questions of {F}atou and {E}remenko.
\newblock {\em Proc. Amer. Math. Soc. 133}, 4 (2005), 1119--1126.

\bibitem{RS}
{\sc Rippon, P.~J., and Stallard, G.~M.}
\newblock Escaping points of meromorphic functions with a finite number of poles.
\newblock {\em J. Anal. Math. 96}, (2005), 225--245.

\bibitem{pre05533139}
{\sc Rippon, P.~J., and Stallard, G.~M.}
\newblock {Escaping points of entire functions of small growth.}
\newblock {\em Math. Z. 261}, 3 (2009), 557--570.

\bibitem{MR2801622}
{\sc Rippon, P.~J., and Stallard, G.~M.}
\newblock Boundaries of escaping {F}atou components.
\newblock {\em Proc. Amer. Math. Soc. 139}, 8 (2011), 2807--2820.

\bibitem{MR2792984}
{\sc Rippon, P.~J., and Stallard, G.~M.}
\newblock Slow escaping points of meromorphic functions.
\newblock {\em Trans. Amer. Math. Soc. 363}, 8 (2011), 4171--4201.

\bibitem{Rippon01102012}
{\sc Rippon, P.~J., and Stallard, G.~M.}
\newblock Fast escaping points of entire functions.
\newblock {\em Proc. London Math. Soc. (3) 105}, 4 (2012), 787--820.

\bibitem{MR1956142}
{\sc D.~Schleicher and J.~Zimmer.}
\newblock Escaping points of exponential maps.
\newblock {\em J. London Math. Soc. (2)}, 67(2):380--400, 2003.

\bibitem{DS1}
{\sc Sixsmith, D.}
\newblock Entire functions for which the escaping set is a spider's web.
\newblock {\em Math. Proc. Cambridge Philos. Soc.}, 151(3):551--571, 2011.

\bibitem{Sixsmithmax}
{\sc Sixsmith, D.}
\newblock Maximally and non-maximally fast escaping points of transcendental
  entire functions.
\newblock {\em Math. Proc. Cambridge Philos. Soc.}, 158(2):365--383, 2015.

\bibitem{MR1260113}
{\sc Stallard, G.~M.}
\newblock The {H}ausdorff dimension of {J}ulia sets of meromorphic functions.
\newblock {\em J. Lond. Math. Soc. (2) 49}, 2 (1994), 281--295.

\bibitem{MR1357062}
{\sc Stallard, G.~M.}
\newblock The {H}ausdorff dimension of {J}ulia sets of entire functions. {II}.
\newblock {\em Math. Proc. Cambridge Philos. Soc. 119}, 3 (1996), 513--536.

\end{thebibliography}

\end{document}